\documentclass[twoside,a4paper,11pt]{amsart}
\usepackage{amsmath,amssymb,amsthm,amscd}
\usepackage[]{latexsym,amssymb,amsmath,amsfonts, amsthm, verbatim}
\usepackage[all,cmtip,ps]{xy}
\usepackage[latin1]{inputenc}
\usepackage[english]{babel}
\usepackage{color}
\usepackage[dvipsnames]{xcolor}

\usepackage{graphicx}

\newcommand{\Add}{\mathrm{Add}}

\newcommand{\rad}{\mathrm{rad}}

\DeclareMathOperator{\Hom}{Hom} \DeclareMathOperator{\End}{End}
\DeclareMathOperator{\Ext}{Ext} \DeclareMathOperator{\Tor}{Tor}
\DeclareMathOperator{\Ker}{Ker} \DeclareMathOperator{\Img}{Im}
\DeclareMathOperator{\Coker}{Coker}

\newcommand{\Dcal}{\ensuremath{\mathcal{D}}}
\newcommand{\Scal}{\ensuremath{\mathcal{S}}}

\newcommand{\Ccal}{\ensuremath{\mathcal{C}}}

\newcommand{\cc}{\Ccal}
\newcommand{\cs}{\Scal}
\newcommand{\cd}{\Dcal}

\newcommand{\Mod}{\mathrm{Mod}}

\theoremstyle{plain}
\newtheorem{thm}{Theorem}

\newtheorem{prop}{Proposition}[section]
\newtheorem{lem}[prop]{Lemma}
\newtheorem{lemma}[prop]{Lemma}
\newtheorem{cor}[prop]{Corollary}

\newtheorem{Que}[prop]{Question}
\theoremstyle{definition}
\newtheorem{ex}[prop]{Example}
\newtheorem*{ex*}{Example}

\theoremstyle{remark}
\newtheorem*{rem}{Remark}

\newcommand{\ra}{\rightarrow}

\newcommand{\ten}{\otimes}
\newcommand{\lten}{\overset{\mathbf{L}}{\ten}}
\newcommand{\rhom}{\mathbf{R}\mathrm{Hom}}

\newcommand{\mcx}{\mathcal{X}}

\begin{document}

\date{\today}



\begin{center}

{\large \bf Recollements and stratifying ideals}

\bigskip

{\sc Lidia Angeleri  H\" ugel\footnote{LAH acknowledges partial
support by    Fondazione Cariparo, Progetto di Eccellenza ASATA.
}, Steffen Koenig, Qunhua Liu\footnote{QL acknowledges support by Natural Science Foundation of China 11301272 and of Jiangsu Province BK20130899.}, Dong Yang\footnote{DY acknowledges support by the DFG priority programme SPP 1388 through grants YA297/1-1 and KO1281/9-1 and by Natural Science Foundation of China 11301272.}}
\bigskip
\end{center}

\bigskip

\address{Lidia Angeleri  H\" ugel
\\ Dipartimento di Informatica - Settore Matematica
\\ Universit\`a degli Studi di Verona
\\ Strada Le Grazie 15 - Ca' Vignal 2
\\ I - 37134 Verona, Italy}
\email{lidia.angeleri@univr.it}

\address{Steffen Koenig\\
Institute of Algebra and Number Theory,
University of Stuttgart \\ Pfaffenwaldring 57 \\ 70569 Stuttgart,
Germany} \email{skoenig@mathematik.uni-stuttgart.de}

\address{Qunhua Liu \\ Institute of Mathematics, School of Mathematical Sciences, Nanjing Normal University \\  Nanjing 210023,
P.R.China}
\email{05402@njnu.edu.cn}

\address{Dong Yang \\ Department of Mathematics, Nanjing University \\ Nanjing 210093, P. R. China}
\email{yangdong@nju.edu.cn}

\begin{quote}
\tiny
{\sc Abstract.} Surjective homological epimorphisms with stratifying kernel can be used to construct recollements of derived module categories. These `stratifying' recollements are derived from recollements of module categories.
Can every recollement be put in this form, up to equivalence? A negative answer will be given after providing a characterisation of recollements  equivalent to stratifying ones.
Moreover, criteria for a ring epimorphism to be `stratifying' will be presented as well as constructions of such epimorphisms.
\\
{\sc MSC 2010 classification:} 16E35, 18E30 \\
{\sc Key words:} Derived module category, recollement, homological epimorphism, stratifying ideal.
\end{quote}

\bigskip

\section{Introduction}

Recollements of triangulated categories have been introduced by
Beilinson, Bernstein and Deligne \cite{BBD} in order to deconstruct
a derived category of constructible sheaves into an open and a closed
part. This concept is meaningful also for derived module
categories of rings. A recollement
\[
\xy (-46,0)*{\mathcal \Dcal(B)}; {\ar (-25,2)*{}; (-35,2)*{}};
{\ar (-35,0)*{}; (-25,0)*{}^{}}; {\ar(-25,-2)*{}; (-35,-2)*{}};
(-15,0)*{{\mathcal D}(A)}; {\ar (5,2)*{}; (-5,2)*{}_{}};
{\ar (-5,0)*{}; (5,0)*{}}; {\ar (5,-2)*{}; (-5,-2)*{}};
(16,0)*{{\mathcal \Dcal(C)}};
\endxy
\]
can be viewed as a short exact sequence of derived module categories
of rings $A$, $B$ and $C$,
with the given derived category $\Dcal(A)$ as middle term.

Directly translating recollements from categories of perverse
sheaves on
flag manifolds to block algebras $A$ of the Bernstein-Gelfand-Gelfand
category $\mathcal O$ of a semisimple complex Lie algebra produces recollements of
derived module categories of quasi-hereditary, or more generally
stratified, algebras, as introduced and studied by Cline, Parshall
and Scott \cite{CPS88,CPS88b,CPS}. By \cite{CPS}, a stratifying ideal $AeA$ of a ring $A$ and the associated  ring epimorphism $A \rightarrow A/AeA$
induce a
recollement of derived module categories of the special form
\[
\xy (-48,0)*{\mathcal \Dcal(A/AeA)}; {\ar (-25,2)*{}; (-35,2)*{}};
{\ar (-35,0)*{}; (-25,0)*{}^{}}; {\ar(-25,-2)*{}; (-35,-2)*{}};
(-15,0)*{{\mathcal D}(A)}; {\ar (5,2)*{}; (-5,2)*{}_{}};
{\ar (-5,0)*{}; (5,0)*{}}; {\ar (5,-2)*{}; (-5,-2)*{}};
(16,0)*{{\mathcal \Dcal(eAe)}.};
\endxy
\]

The ring epimorphism $A \rightarrow A/AeA$ giving rise to this recollement
has the additional property that it is a homological epimorphism and its
kernel is a stratifying ideal. Most examples of recollements in the literature are
known to be of this form, up to applying derived equivalences to the three rings $A$, $B$ and $C$.

One of the main applications of recollements is to relate homological data
of the three rings, such as global or finitistic dimension \cite{Happel93,AKLY2}, K-theory \cite{Schlichting06,CX,AKLY2}
and Hochschild
(co)homology \cite{Keller98,Han11,KoenigNagase09}. From a practical point of view, for a recollement induced by a stratifying
ideal, the resulting long exact sequences are much easier to handle, one reason
being that the six functors in this case are derived from the obvious
six functors associated with an idempotent on the level of module categories.
Moreover, up to Morita equivalence of the algebras involved,
these are exactly the recollements that can be produced by deriving recollements of module
categories (see the classification of such recollements of abelian categories by
Psaroudakis and Vit\'oria \cite{PV}).

Motivated by a kind of folklore conjecture, we consider the question if all recollements
of derived module categories are of this form. Taken naively, this question has
an obvious negative answer, since one may hide the epimorphism $A \rightarrow A/AeA$ by
replacing, for instance, $\Dcal(B)$ by an equivalent derived category $\Dcal(B')$,
where there is no morphism at all from $A$ to $B'$. Moreover, there exist injective
homological epimorphisms, which also induce recollements. A meaningful way to
formulate the question is the following one, suggested by Changchang Xi:
\smallskip

\begin{Que} \label{mainquestion}
Given a recollement relating the derived module categories of
three rings $A$, $B$ and $C$ as
above, is there another - equivalent - recollement, obtained by replacing the
derived module categories by equivalent ones, that is induced by a stratifying
ideal?
\end{Que}


In the following, we will call a recollement `stratifying' when it is induced by
a stratifying ideal, and thus derived from a recollement on module level. Then the question is whether `stratifying' recollements do give a normal form of recollements, similar to the situation for module categories described in \cite{PV}.

In the second Section, we will see that the answer to Question \ref{mainquestion} is negative, even for finite dimensional algebras of finite global dimension, and even when allowing to change all three derived categories.
More precisely, we are going to characterise - in Theorem \ref{maintheorem} and  Corollary \ref{maintheorem-variation} -  the recollements, which up
to derived equivalence are induced by a stratifying ideal. Using these
characterisations, we will give a counterexample to Question \ref{mainquestion}.

In parallel work \cite{PVpositive}, Psaroudakis and Vit\'oria have been able
to provide a positive answer to this question for hereditary rings, using
rather different methods.
\medskip

The negative answer to Question \ref{mainquestion} puts additional emphasis on the following questions, which we are going to address in Sections 3 and 4.

\begin{Que} \label{whenisepisurjective} Given a ring epimorphism $f: A \rightarrow B$, when is it a surjective homological epimorphism with a stratifying kernel?
\end{Que}

The question of when an abstractly given ring epimorphism $f$ is surjective has been studied by several authors in the past, see e.g.~\cite{Reid,Storrer}. In Section 3 we provide some new criteria. The main result is Theorem \ref{surjepi}, which gives a sufficient condition for surjectivity, and then also for the other desired properties, when $A$ is a perfect (e.g.~an artinian) ring.

\begin{Que} \label{howtoconstructstratifyingepi} Given a ring epimorphism that is not stratifying,
when can it be replaced by a surjective homological epimorphism with a stratifying kernel?
\end{Que}

In Section 4, we give two such constructions.
The first one uses a connection with tilting theory from \cite{GL,AHS} to show that certain recollements induced by  injective ring epimorphisms are equivalent to stratifying recollements. The second construction, Theorem \ref{secondconstruction}, allows to form a ring epimorphism $A\to C$ with favourable properties from a given ring epimorphism $A\to B$.

Throughout the paper, rings are associative and unital. Modules by default are right modules and $\Dcal(A)$ denotes the unbounded derived category of the category of all modules over a ring $A$.

\section{Characterisations, and a counterexample}

In this Section we are going to characterise recollements equivalent to stratifying ones, see Theorem \ref{maintheorem} and Corollary \ref{maintheorem-variation}.
As a consequence, we will answer Question \ref{mainquestion}, and variations of it, negatively, by giving an explicit counterexample.
\smallskip

\subsection{Definitions and notations}\label{basics}

We will use standard terminology for derived categories, derived
equivalences and tilting complexes, see for instance \cite{Zimmermann}.

Let $\cc$ be a triangulated category with shift functor $[1]$.
An object $X$ of $\cc$ is \emph{exceptional} if $\Hom_{\cc}(X,
X[n])=0$ unless $n=0$. Let $\cs$ be a set of objects of $\cc$. As usual,
thick$\,\cs$ denotes
the smallest triangulated subcategory of $\cc$
containing $\cs$ and closed under taking direct summands.
Assume further that $\cc$ has all (set-indexed) infinite direct
sums. An object $X$ of $\cc$ is \emph{compact} if the functor
$\Hom_\cc(X,-)$ commutes with taking direct sums.

A \emph{recollement}~\cite{BBD} of
triangulated categories is a diagram
\begin{eqnarray}\xymatrix@!=3pc{\mathcal{C}'
\ar[r]|{i_*=i_!} &\mathcal{C} \ar@<+2.5ex>[l]|{i^!}
\ar@<-2.5ex>[l]|{i^*} \ar[r]|{j^!=j^*} &
\mathcal{C}''\ar@<+2.5ex>[l]|{j_*}
\ar@<-2.5ex>[l]|{j_!}}\label{eq:recollement}\end{eqnarray}of
triangulated categories and triangle functors such that
\begin{enumerate}
\item $(i^\ast,i_\ast)$,\,$(i_!,i^!)$,\,$(j_!,j^!)$,\,$(j^\ast,j_\ast)$
are adjoint pairs;

\item  $i_\ast,\,j_\ast,\,j_!$  are full embeddings;

\item  $i^!\circ j_\ast=0$ (and thus also $j^!\circ i_!=0$ and
$i^\ast\circ j_!=0$);

\item  for each $C\in \mathcal{C}$ there are triangles
\[
\xymatrix@R=0.5pc{
i_! i^!(C)\ar[r]& C\ar[r]& j_\ast j^\ast (C)\ar[r]& i_!i^!(A)[1]\\
j_! j^! (C)\ar[r]& C\ar[r]& i_\ast i^\ast(C)\ar[r]& j_!j^!(A)[1]
}
\]
where the maps are given by adjunctions.
\end{enumerate}
Thanks to (1) and (3), the two triangles (often called the
\emph{canonical triangles}) in (4) are unique up to  unique isomorphisms.


Two recollements involving categories $\Ccal', \Ccal, \Ccal''$ and $\Dcal', \Dcal, \Dcal''$, respectively, are called \emph{equivalent}, if there exists a triangle equivalence $\Ccal \simeq \Dcal$ inducing triangle equivalences $\Ccal' \simeq \Dcal'$ and $\Ccal'' \simeq \Dcal''$ such that all squares commute.
\smallskip

An epimorphism $\varphi:A\ra B$ in the category of rings is called a {\em ring
  epimorphism}. Equivalently, the induced functor
\mbox{$\varphi_*: \text{Mod-}B \ra \text{Mod-}A$}
is a full embedding. Another equivalent characterisation of $\varphi$ being a ring epimorphism is that $\Coker(\varphi) \otimes_A B = 0$,
see e.g. \cite[Chapter XI, Proposition 1.2]{Ste}.

Furthermore $\varphi$ is a {\em
homological epimorphism} if and only if the induced functor
\mbox{$\varphi_*: \Dcal(B) \ra \Dcal(A)$} is a full embedding,
or equivalently,  $\varphi$ is a ring epimorphism with $\Tor^A_i(B,B)=0$ for all $i\ge 1$ (cf. \cite[Theorem
4.4]{GL}). In this case, $\varphi$ induces a recollement
\[
\xy (-53,0)*{\Dcal(B)}; {\ar (-28,3)*{}; (-44,3)*{}_{i^{\ast}}}; {\ar
(-44,0)*{}; (-28,0)*{}|{ i_{\ast}=i_!}}; {\ar(-28,-3)*{};
(-44,-3)*{}^{i^!}}; (-19,0)*{\Dcal(A)}; {\ar (3,3)*{}; (-11,3)*{}};
{\ar (-11,0)*{}; (3,0)*{}}; {\ar(3,-3)*{}; (-11,-3)*{}};
(8,0)*{\mcx};
\endxy
\]
for some triangulated
category $\mcx$, and the functors on the left hand
side are induced by $\varphi$, that is, $i^* = -\lten_A B$, $i^! =
\rhom_A(B,-)$ and $i_* = \varphi_*$.

\smallskip

Homological epimorphisms starting in a ring $A$ are closely related with recollements of $\Dcal(A)$ where the left hand term is a derived category of a ring too. We say that a recollement  of $\Dcal(A)$ by triangulated categories $\Ccal'$ and $\Ccal''$ is {\it induced by a homological epimorphism} $\varphi:A\to B$ if there is an equivalence  $F:\Ccal'\to \Dcal(B)$ such that $i_\ast=\varphi_\ast\circ F$.
The following result characterises such recollements.
\begin{prop}\cite[1.7]{AKL1}\label{akl}
A recollement of $\Dcal(A)$  is  induced by some homological epimorphism $A\to B$ if and only if $i^{\ast}(A)$ is exceptional. In this case, $B$ is the endomorphism ring of $i^{\ast}(A)$.
\end{prop}

A surjective homomorphism $\varphi: A \ra B$ is an epimorphism. If it is
homological and its kernel $I$ is of the form $AeA$ with $e=e^2 \in A$ an
idempotent, then $I=AeA$ is called a {\em stratifying ideal}, see \cite{CPS}. For the purpose of this article, we then call the induced recollement a {\em stratifying recollement}. It is of the following form
\[
\xy (-46,0)*{\mathcal \Dcal(A/AeA)}; {\ar (-19,3)*{}; (-35,3)*{}_{i^\ast}};
{\ar (-35,0)*{}; (-19,0)*{}|{i_\ast=i_!}}; {\ar(-19,-3)*{}; (-35,-3)*{}^{i^!}};
(-12,0)*{{\mathcal D}(A)}; {\ar (11,3)*{}; (-5,3)*{}_{j_!}};
{\ar (-5,0)*{}; (11,0)*{}|{j^*=j^!}}; {\ar (11,-3)*{}; (-5,-3)^{j_*}};
(20,0)*{{\mathcal \Dcal(eAe)}};
\endxy
\]
where
\[
\begin{array}{c}
i^*=-\lten_A A/AeA,\ \ i^!=\rhom_A(A/AeA,-),\\
i_*=\rhom_{A/AeA}(A/AeA,-)=-\lten_{A/AeA}A/AeA=i_!,\\
j_!=-\lten_{eAe} eA,\ \  j_*=\rhom_{eAe}(Ae,-), \\
j^!=\rhom_A(eA,-)=-\lten_A Ae=j^*.
\end{array}
\]
Using this language, Question \ref{mainquestion} asks whether each recollement of derived module categories of rings is equivalent to a stratifying one.

\smallskip

At this point, it has to be noted that homological epimorphisms do not behave well under Morita or derived equivalences applied to the corresponding recollement. The following easy example shows that when formulating Question \ref{mainquestion} it is necessary to allow for changes of the data by Morita or derived equivalences.

Let $A=M_2(k)\times M_3(k)$, $B=M_2(k)$, and $\lambda:A\ra B$ the projection.
Then $\lambda$ is a homological epimorphism with stratifying kernel,
inducing a recollement with $B$ on the left hand side, $A$ in the middle and
$M_3(k)$ on the right hand side. Let $B'=k$, which is Morita equivalent to $B$. Then there is no ring homomorphism from $A$ to $B'$ at all, and in particular no homological epimorphism with stratifying kernel.

\subsection{Characterisations of stratifying recollements}

In the first answer to Question \ref{mainquestion}, we keep the ring $C$ fixed
and characterise when it can be realised as
the ring $eA'e$ in a stratifying recollement
equivalent to a given one:

\begin{thm} \label{maintheorem}
Fix rings $A$, $B$ and $C$ and a recollement
\[
\xy (-48,0)*{\mathcal \Dcal(B)}; {\ar (-25,2)*{}; (-35,2)*{}};
{\ar (-35,0)*{}; (-25,0)*{}^{}}; {\ar(-25,-2)*{}; (-35,-2)*{}};
(-15,0)*{{\mathcal D}(A)}; {\ar (5,2)*{}; (-5,2)*{}_{}};
{\ar (-5,0)*{}; (5,0)*{}}; {\ar (5,-2)*{}; (-5,-2)*{}};
(16,0)*{{\mathcal \Dcal(C)}.};
\endxy
\tag{$R$}
\]
Then the following statements are equivalent.

(a) There exist rings $A'$ and $B'$ that are derived equivalent to
$A$ and $B$ respectively, and an idempotent $e \in A'$ such that
$C = eA'e$, $B' = A'/A'eA'$,
and the projection $\pi: A' \rightarrow B'$ is a
homological epimorphism with stratifying kernel,
which induces a  recollement equivalent to (R).

(b) The complex $j_!(C)$ is a direct summand of a tilting complex
$T$ over $A$ such that $i^{\ast}(T)$ is exceptional.
\end{thm}

\begin{proof}
Suppose (a) is given. Then the homological epimorphism $\pi$ induces a stratifying recollement
\[
\xy (-48,0)*{\mathcal \Dcal(B')}; {\ar (-25,2)*{}; (-35,2)*{}};
{\ar (-35,0)*{}; (-25,0)*{}^{}}; {\ar(-25,-2)*{}; (-35,-2)*{}};
(-15,0)*{{\mathcal D}(A')}; {\ar (5,2)*{}; (-5,2)*{}_{}};
{\ar (-5,0)*{}; (5,0)*{}}; {\ar (5,-2)*{}; (-5,-2)*{}};
(16,0)*{{\mathcal \Dcal(C)}.};
\endxy
\]
where $C = eA'e$ and $B' = A'/A'eA'$. The functor $j_!$ is the derived tensor functor
$- \lten_{eA'e} e A'$, which sends $C$ to the
projective module $j_!(C) = e A'$. Setting $T := A'$ shows that $j_!(C)$
is a direct summand of a tilting complex. The functor $i^{\ast}$ is
the derived tensor functor $- \stackrel{\mathbf L}{\otimes}_{A'} B'$, which
sends $A'$ to $i^{\ast}(A')=B'$, which is exceptional. Thus the recollement
induced by $\pi$ satisfies the conditions in (b). Moving from $A'$ and $B'$
to $A$ and $B$, respectively, by derived equivalences, does
not affect the conditions in (b), since a derived equivalence sends a
tilting complex to a tilting complex and an exceptional object to an
exceptional object. Hence the original recollement satisfies (b) as well.
\smallskip

Suppose (b) is satisfied. Set $A' := \End_{\cd(A)}(T)$. Write $T = T_1 \oplus T_2$, where
$T_1 = j_!(C)$. The tilting complex $T$ induces an equivalence
$\alpha: \Dcal(A) \xrightarrow{\sim} \Dcal(A')$, which we use to change the recollement into
one with middle term $\Dcal(A')$. The new functor $j_!$ sends $C$ to the image
of $T_1$ under the derived equivalence $\alpha$. By construction of $\alpha$, this image $j_!(C)=\alpha(T_1)$ is a
projective module $eA'$ for some idempotent $e$. Composing $j_!$  with a
derived auto-equivalence of $C$, if necessary, we may assume that $C=eA'e$ and the new $j_!$ is
the derived tensor functor $- \stackrel{\mathbf L}{\otimes}_{eA'e} e A'$.
Therefore, the other two functors on the right hand side of the recollement,
which are uniquely determined by being adjoints, are as required in a
stratifying recollement, too.

The ring $A'$ may in general not be a flat algebra over $\mathbb{Z}$. Therefore, we are now going to replace it by a flat dg ring $A''$, chosen as follows:
Let $f:A''\rightarrow A'$ be a cofibrant replacement of $A'$. This is provided as part of the model structure of the category of small dg categories constructed by Tabuada in \cite{Tabuada05}. Here the base ring is $\mathbb{Z}$. As shown in the proof of \cite[Lemma 5]{NS1}, $A''$ is flat over $\mathbb{Z}$.
By definition, $f$ is a quasi-equivalence in the sense of \cite[Section 7]{Keller}. In particular, the 0-cohomology of $A''$ is $A'$, and $f$ induces a derived equivalence $\cd(A'')\rightarrow \cd(A')$. Thus we obtain a recollement which has $\cd(A'')$ as the middle term and which is equivalent to the original one. By \cite[Theorem 4]{NS1} (which requires the dg ring $A''$ to be flat over $\mathbb{Z}$), the new recollement is induced by a homological epimorphism of dg rings $\varphi:A''\rightarrow B''$, where $B''$ is a dg endomorphism ring of $i^*(A'')$.

By construction of the new
recollement, $i^{\ast}(A'')$ equals what was $i^{\ast}(T)$ in the old
recollement. Therefore, $i^{\ast}(A'')$ is exceptional. As a consequence,
the dg endomorphism ring $B''$ is quasi-isomorphic to its $0$-cohomology $B':=H^0(B'')$, which is
an ordinary ring.
Returning from $A''$ to its derived equivalent  0-cohomology $A'$, we can replace the last recollement by an equivalent one with  left hand term $\cd(B')$, middle term $\cd(A')$, and right hand term unchanged.
Let $\pi: A'\ra B'$ be the $0$-cohomology of the homological epimorphism $\varphi:A''\ra B''$. Then $\pi$ is a homological epimorphism,
and the  triangle
$$A'e \lten_{eA'e} e A' \rightarrow A' \stackrel{\pi}
{\rightarrow} B' \rightarrow A'e \lten_{eA'e} e A'[1]$$
yields a short exact sequence
\[
\xymatrix{
0 \ar[r] & A'e \otimes_{eA'e} e A' \ar[r] & A' \ar[r]^{\pi} & B' \ar[r] & 0,}
\]
since both $H^1(A'e\lten_{eA'e} eA')$ and $H^{-1}(B'')$ vanish. Here, the kernel of $\pi$ is the multiplication map $A'e \otimes_{eA'e} eA' \ra A'$, whose image is $A'eA'$. Therefore $B'\simeq A/A'eA'$ and up to this isomorphism $\pi$ is identified with the quotient map $A'\ra A/A'eA'$, implying that the ideal $A'eA'$ is stratifying. \smallskip

The sequence of modifications to the given recollement
is summarised in the following diagram
\small
\begin{eqnarray*}
\xymatrix@C=3pc@R=1.3pc{
\Dcal(B) \ar@{=}[d]
\ar[r] &\Dcal(A) \ar@<+1.2ex>[l]
\ar@<-1.2ex>[l] \ar[r] \ar[d]^{\simeq} &
\Dcal(C) \ar@<+1.2ex>[l] \ar@<-1.2ex>[l] \ar@{=}[d] \\
\Dcal(B) \ar[d]^{\simeq}
\ar[r] &\Dcal(A') \ar@<+1.2ex>[l]
\ar@<-1.2ex>[l] \ar[r] \ar[d]^\simeq &
\Dcal(eA'e) \ar@<+1.2ex>[l] \ar@<-1.2ex>[l] \ar@{=}[d] \\
\Dcal(B'') \ar[d]^{\simeq}
\ar[r] &\Dcal(A'') \ar@<+1.2ex>[l]
\ar@<-1.2ex>[l] \ar[r] \ar[d]^\simeq &
\Dcal(eA'e) \ar@<+1.2ex>[l] \ar@<-1.2ex>[l] \ar@{=}[d] \\
\Dcal(B') \ar[d]^{\simeq}
\ar[r] &\Dcal(A') \ar@<+1.2ex>[l]
\ar@<-1.2ex>[l] \ar[r] \ar@{=}[d] &
\Dcal(eA'e) \ar@<+1.2ex>[l] \ar@<-1.2ex>[l] \ar@{=}[d] \\
\Dcal(A'/A'eA') \ar[r] &\Dcal(A') \ar@<+1.2ex>[l]
\ar@<-1.2ex>[l] \ar[r] &
\Dcal(eA'e). \ar@<+1.2ex>[l] \ar@<-1.2ex>[l]
}
\end{eqnarray*}
\end{proof}

When we relax the condition on $C$ to being not necessarily isomorphic, but
at least derived equivalent to $eA'e$ in the stratifying recollement, the
answer to Question \ref{mainquestion} is as follows:

\begin{cor} \label{maintheorem-variation}
Fix rings $A$, $B$ and $C$ and a recollement
\[
\xy (-48,0)*{\mathcal \Dcal(B)}; {\ar (-25,2)*{}; (-35,2)*{}};
{\ar (-35,0)*{}; (-25,0)*{}^{}}; {\ar(-25,-2)*{}; (-35,-2)*{}};
(-15,0)*{{\mathcal D}(A)}; {\ar (5,2)*{}; (-5,2)*{}_{}};
{\ar (-5,0)*{}; (5,0)*{}}; {\ar (5,-2)*{}; (-5,-2)*{}};
(16,0)*{{\mathcal \Dcal(C)}.};
\endxy
\tag{$R$}
\]
Then the following statements are equivalent.

(a) There exist rings $A'$, $B'$ and $C'$ that are derived equivalent to
$A$, $B$ and $C$ respectively, and an idempotent $e \in A'$ such that
$C' = eA'e$, $B' = A'/A'eA'$ and the projection $\pi: A' \rightarrow B'$ is a
homological epimorphism with stratifying kernel
which induces a recollement  equivalent to (R).

(b) There exists a tilting complex $T_0$ over $C$
such that the
complex $j_!(T_0)$ is a direct summand of a tilting complex $T$
over $A$ with $i^{\ast}(T)$ being exceptional.
\end{cor}
\begin{proof} Denote $\End_{\cd(C)}(T_0)$ by $C'$.
The tilting complex $T_0$ in (b) induces an equivalence of $\Dcal(C)$ with
$\Dcal(C')$. Using this equivalence, the given recollement can be changed into one with  $\Dcal(C')$ on the right hand side. Applying Theorem \ref{maintheorem} to this recollement proves the equivalence of (a) and (b).
\end{proof}

\begin{rem} In special situations, part of condition (b) may be dropped. Here is an example:
Suppose $A$ is a finite-dimensional algebra over a field with only two isomorphism classes of simple modules.
Then (b) is equivalent to
\begin{itemize}
\item[(b')] the complex $j_!(C)$ can be completed to a tilting complex $T$ over $A$.
\end{itemize}
In fact, in this case, both $B$ and $C$ are local algebras by \cite[Proposition 6.5]{AKLY2}. Therefore, any tilting complex $T_0$ over $C$ is a projective generator, so $j_!(T_0)$ can be completed to a tilting complex if and only if $j_!(C)$ can be completed to a tilting complex. Now assume that $j_!(C)$ can be completed to a tilting complex $T$. Then $i^*(T)$, being compact in $\cd(B)$, either is a shifted projective module or it has self-extensions in positive degrees, see for instance \cite[2.11-2.13]{RZ}.
But by \cite[Proposition 6.6]{AKLY2} (or the more general \cite[Theorem 4.8]{IyamaYang14}), $i^*(T)$ is a silting object of $K^b(\mathrm{proj} A)/\text{{thick}}{\,j_!(C)}\cong K^b(\mathrm{proj} B)$, so the latter case does not occur. That is, $i^*(T)$ is exceptional.
\end{rem}

\begin{rem}
 In \cite{Koenig91}, the existence of recollements of derived module categories has been characterised in terms of the existence of two exceptional complexes satisfying certain orthogonality conditions. The complex $j_!(C)$ always is exceptional. If $i_{\ast}(i^{\ast}(T))$ is exceptional, too, these two complexes together satisfy the conditions in the characterisation. Thus, exceptionality of $i_{\ast}(i^{\ast}(T))$, which corresponds to $i^{\ast}(T)$ being exceptional, can be understood as restating the existence of the recollement.  When taking this point of view, the additional condition needed for this recollement to be  stratifying (up to equivalence)
 is that $j_!(C)$ can be completed to a tilting complex.
\end{rem}

\begin{rem}
Given $A$ and $C =eAe$, the exact functor $- \cdot e$ can be used to construct
a `half recollement', which is the right hand side (involving $A$ and $C$) of the stratifying recollement investigated here. The left hand side then can be completed by taking
 the derived category of some dg ring. The problem, however, is to construct the left hand side as the derived category of an ordinary ring. This is not always possible. There do exist examples of recollements, with given $A$ and $C=eAe$, where the left hand side cannot be a derived module category. This happens for instance, if $A$ has finite global dimension, but the endomorphism ring $C$ of some exceptional (or even projective) object has infinite global dimension,
 see e.g.~\cite[Proposition 2.14]{AKLY2}.
\end{rem}

\subsection{A counterexample}

Here is an example of a recollement that cannot be turned into a stratifying one
by replacing $A$, $B$ and $C$ by derived equivalent algebras. In other words, the following recollement does not satisfy condition (b) in Theorem \ref{maintheorem} nor in Corollary \ref{maintheorem-variation}.

\begin{ex}

In \cite[Example 4.4]{LVY}, the following algebra is studied:

Let $k$ be a field and let $A$ be the $k$-algebra given by quiver and
relations
\[\xymatrix{&&2\ar@<.7ex>[dd]^{\gamma}\ar@<-.7ex>[dd]_{\beta}\\
1\ar[rru]^{\alpha}&&&,\\
&& 3\ar[llu]^{\delta}}\hspace{10pt}\xymatrix{\\
\beta\alpha =0,~~ \alpha\delta=0,~~\delta\gamma=0.}\]
The simple module $S_1$ supported at $1$ is a compact exceptional
module of projective dimension $2$. It has a minimal projective resolution over $A$
given by the exact sequence
\[\xymatrix{0\ar[r] &P_2\ar[r]^{\beta} & P_3\ar[r]^{\delta} & {P_1}\ar[r] &
S_1\ar[r]&0}.\]

As shown in \cite{LVY}, setting $e = e_2 + e_3$ the algebra $A$ has a
stratifying ideal $AeA$ and thus a stratifying recollement
\[
\xy (-48,0)*{\mathcal \Dcal(A/AeA)}; {\ar (-25,2)*{}; (-35,2)*{}};
{\ar (-35,0)*{}; (-25,0)*{}^{}}; {\ar(-25,-2)*{}; (-35,-2)*{}};
(-15,0)*{{\mathcal D}(A)}; {\ar (5,2)*{}; (-5,2)*{}_{}};
{\ar (-5,0)*{}; (5,0)*{}}; {\ar (5,-2)*{}; (-5,-2)*{}};
(16,0)*{{\mathcal \Dcal(eAe)}.};
\endxy
\]
where $A/AeA$ is one-dimensional, \emph{i.e.} isomorphic to the ground field $k$,
and as a right $A$-module isomorphic to the simple module $S_1$. The
algebra $eAe$ is isomorphic to the Kronecker algebra, hence hereditary.
The algebra $A$ has finite global dimension, and therefore it is possible
to mutate (that is, extend) the above recollement downwards, by \cite[Section 3]{AKLY2}. Thus, there is a recollement
\[
\xy (-48,0)*{\mathcal \Dcal(B=eAe)}; {\ar (-25,2)*{}; (-35,2)*{}};
{\ar (-35,0)*{}; (-25,0)*{}^{}}; {\ar(-25,-2)*{}; (-35,-2)*{}};
(-15,0)*{{\mathcal D}(A)}; {\ar (5,2)*{}; (-5,2)*{}_{}};
{\ar (-5,0)*{}; (5,0)*{}}; {\ar (5,-2)*{}; (-5,-2)*{}};
(16,0)*{{\mathcal \Dcal(C=k)}.};
\endxy
\]
where $j_!(C)$ equals $S_1$. Extending earlier work of Rickard and Schofield,
it has been checked in \cite{LVY} that $S_1$ cannot be a direct summand of a
tilting complex, that is, $j_!(C)$ fails the condition in (b).
Hence this recollement cannot be turned into a stratifying one
by changing $A$ and $B$. Moreover, the auto-equivalences of $\Dcal(k)$ are
compositions of Morita equivalences and shifts. Therefore, replacing
$C$ by a derived equivalent algebra $C'$ does not remove the obstruction
to extending $j_!(T_0)$.

It follows that condition (b) in Theorem \ref{maintheorem} and also in Corollary \ref{maintheorem-variation} fails.

Note that $i^{\ast}(A)$ as a module over the Kronecker algebra $eAe$ is a
direct sum of projective modules and a quasi-simple regular module $M$. Since
$M$ has self-extensions, $i^{\ast}(A)$ is not exceptional. Therefore, by
Proposition~\ref{akl}, this recollement is not induced by a homological
epimorphism.

This recollement also restricts to recollements on the level of bounded or left
or right bounded derived categories
by \cite[Proposition 4.12]{AKLY2}. Homotopy categories of projectives in this case are covered, too, since they coincide with the bounded derived categories. So, Question \ref{mainquestion} has a negative answer
for all these choices of derived module categories.
\end{ex}

\section{Surjective homological epimorphisms and stratifying
ideals}

When is a ring epimorphism $\varphi:A\to B$ equivalent to a  homological
epimorphism $A \rightarrow A/AeA$ with stratifying kernel? Of course, one first
has to decide if $\varphi$ is surjective. The main result of this Section,
Theorem \ref{surjepi}, provides a
criterion for that. Once surjectivity is known, well-known facts
can be used to decide if the kernel is idempotent or even stratifying.

Recall that a ring $R$ is called {\it semilocal} if the quotient ring $R/\rad(R)$ is semisimple artinian, and it is
{\it right perfect} if in addition the Jacobson radical $\rad(R)$ is a left-t-nilpotent ideal of $R$, i.~e.~for any sequence of elements $a_1,a_2,a_3,\ldots\in \rad(R)$ there is an integer $n>0$ such that $a_n a_{n-1}\ldots a_1=0$.
\begin{thm} \label{surjepi}
Let  $\varphi: A \rightarrow B$ be a ring
epimorphism with $A$  right (or left) perfect and $B$ semilocal.
Suppose that $B$ is basic, that is, $B/\rad(B)$ is a product of
skew-fields.
 Then  $\varphi$ is surjective.
Moreover, $\varphi$ has a stratifying kernel if it is a homological epimorphism.\end{thm}

The crucial point here is to prove the surjectivity of $\varphi$. The proof will use the following characterisation of surjective ring epimorphisms as well as a consequence of this characterisation.

\begin{prop} \label{surjsimples}
Let  $\varphi: A \rightarrow B$ be a ring
epimorphism with $B$ semilocal. Then $\varphi$ is surjective if and only if each simple $B$-module
is simple as an $A$-module.
\end{prop}
\begin{proof}
The only-if-part is clear. To prove the converse, assume that all simple
$B$-modules are simple as $A$-modules, too.
Set $\bar{B}=B/\rad(B)$.
 Clearly, the composition $\pi: A \rightarrow \bar{B}$ of
$\varphi$ with the canonical projection $B\to\bar{B}$
is a ring epimorphism such that all simple $\bar{B}$-modules are simple as $A$-modules, and by Nakayama's lemma
it suffices to show that $\pi$ is  surjective. So we can assume w.l.o.g.~that $B$ is semisimple artinian.

 Suppose now that there is an indecomposable direct summand $S$ of $B$, hence a
simple $B$-module, which is not contained in the image of $\varphi$. Then the
intersection $\Img(\varphi) \cap S$ is a proper $A$-submodule of the simple $B$-module $S$, which by assumption also is a simple $A$-module. Thus,
$\Img(\varphi) \cap S = 0$ and
$S$ is a direct summand of the cokernel of $\varphi$.

As mentioned above in Section 2.1, an equivalent condition of $\varphi$ being a ring
epimorphism is that $\Coker(\varphi) \otimes_A {B} = 0$,
which implies
$S \otimes_A {B} = 0$. But $S \otimes_A B=\varphi^\ast\varphi_\ast(S)\simeq S$, yielding a contradiction. So, $\varphi$ must be surjective.
\end{proof}

The following consequence of Proposition \ref{surjsimples} is a special case of results by Storrer \cite{Storrer}.
Storrer shows (\cite[Corollary 5.4]{Storrer})
that self-injective rings, hence in particular
semisimple rings, are saturated. Here, $R$ saturated means there is no
non-trivial
injective ring epimorphism
starting in $R$.

\begin{lemma} \label{injsemisimple}
Injective ring epimorphisms between semisimple rings are isomorphisms.
\end{lemma}

\begin{proof}
Let $A,B$ be two semisimple rings and $\psi:A\ra B$ an injective ring
epimorphism. Let $S$ be a simple $B$-module. Its endomorphism ring
$\End_B(S)$ is local. Since $\psi$ is a ring epimorphism, the restriction
functor $\psi_{\ast}: \text{Mod-}B \rightarrow \text{Mod-}A$ is fully faithful. Hence,
$S$ must be indecomposable as an $A$-module, and thus simple over $A$.
Now Proposition \ref{surjsimples} can be applied.
\end{proof}

\bigskip

The following known statement will imply properties of $\Ker(\varphi)$ in
Theorem \ref{surjepi}.

\begin{lem} \label{surjepitor}
 Let $\varphi:A\to B$ be a surjective ring epimorphism.
Then $\Tor^A_1(B,B)=0$ if and only if the kernel $\Ker(\varphi)$ is an
idempotent ideal of $A$.

In particular, if $A$ is right (or left) perfect, then every surjective homological
epimorphism  has a stratifying kernel.
\end{lem}

\begin{proof}
The first statement is well known, see e.g.~\cite{BD}.  For the second assertion we use 
 \cite[Proposition 2.1]{Michler}, where it is shown that idempotent ideals of (one-sided) perfect rings are generated by idempotent elements.
\end{proof}

{\bf Proof of Theorem \ref{surjepi}.}

First of all, $\varphi$ factors through its image $C$   as $\varphi=\tau\circ\psi$ where  $\tau:C\hookrightarrow B$  is an injective ring epimorphism, and $\psi:A\twoheadrightarrow C$
is a surjective ring homomorphism, hence an epimorphism, too.
Note that $C$ is again right perfect by \cite[Corollary 11.7.3]{Kasch}. So, we can assume without loss of generality that
 $\varphi$ is injective and
show that it is an isomorphism.

Since the quotient $B/\rad(B)$ is a product of skew-fields, it does not contain non-zero nilpotent elements.
Now  the radical $\rad(A)$ of the right perfect ring $A$ is left-t-nilpotent, thus its elements are nilpotent and so must be their images under $\varphi$. Hence they vanish in $B/\rad(B)$.

Therefore we may pass to the quotients $\bar{A}=A/\rad(A)$ and
$\bar{B}=B/\rad(B)$
and consider the ring homomorphism $\bar{\varphi}:\bar{A}\ra\bar{B}$ between semisimple rings.
It is also a ring epimorphism, for instance because $\Coker{\bar{\varphi}} \otimes_{\bar{A}} \bar{B}=0$. We claim that
$\bar{\varphi}$ is injective. Since $\bar{A}$ is semisimple, the
kernel $\Ker(\bar{\varphi})$ is a direct summand of $\bar{A}$. If it
is not zero, it must contain an idempotent $\bar{e}$. The radical
$\rad(A)$ is left-t-nilpotent, so by \cite[Theorem 11.5.3]{Kasch} there is a lifting $e\in A$ such that $e^2=e$ and
$e+\rad(A)= \bar{e}$. By the choice of $e$, the element $\varphi(e)$ is
an idempotent element belonging to $\rad(B)$, so it is zero, which  implies $e=0$ by the
injectivity of $\varphi$.
Hence also $\bar{e}=0$.
This proves the injectivity of $\bar{\varphi}$.

Now by Lemma \ref{injsemisimple}, $\bar{\varphi}$
is an isomorphism. In particular, the set of simple $B$-modules coincides with
the set of simple $A$-modules. Hence, by Proposition \ref{surjsimples},
$\varphi$ is surjective and thus an isomorphism.
\smallskip


To finish the proof, we just observe that the
last statement follows from Lemma \ref{surjepitor}. \qedsymbol

The proof of Theorem B works as well when relaxing the
  assumption $B$ to be basic by requiring instead an inclusion
  $\tau(\rad(C))\subseteq\rad(B)$.

\medskip

As an application, a positive answer to Question \ref{mainquestion} can be
given in a particular situation:

\begin{cor} \label{stratifyingrec}
  Let $A$ be right (or left) perfect and $B$ semilocal.
  Suppose there is a recollement of the derived module categories
\[
\xy (-45,0)*{\mathcal \Dcal(B)}; {\ar (-25,2)*{}; (-35,2)*{}};
{\ar (-35,0)*{}; (-25,0)*{}^{}}; {\ar(-25,-2)*{}; (-35,-2)*{}};
(-15,0)*{{\mathcal D}(A)}; {\ar (5,2)*{}; (-5,2)*{}_{}};
{\ar (-5,0)*{}; (5,0)*{}}; {\ar (5,-2)*{}; (-5,-2)*{}};
(15,0)*{{\mathcal \Dcal(C)}};
\endxy
\]
such that $i^*(A)$ is exceptional and basic.
Then the recollement is equivalent to a stratifying one.
\end{cor}

Artinian rings, for instance, are perfect and semilocal.

\begin{proof}
By \ref{akl}, the recollement is induced by the homological
epimorphism \mbox{$\varphi:A\ra \End_{\cd(B)}(i^*(A))$.} In order to be able to apply
Theorem \ref{surjepi}, we have to show:

{\em Claim.} The ring $\End_{\cd(B)}(i^{\ast}(A))$ is  semilocal.

{\em Proof.} The complex $X:=i^{\ast}(A)$ is compact. Its entries are finitely
generated projective $B$-modules $X_1, \dots, X_l$ for some $l$.
The endomorphism ring of $X$ as a complex is the subring $R$ of
$R':=\End_B(X_1) \times \dots \times \End_B(X_l)$ formed by $l$-tuples
satisfying the commutativity condition in the definition of morphisms of
complexes. Factoring out homotopies, a quotient ring $\bar{R}$ of $R$ is obtained
that is isomorphic to $\End_{\cd(B)}(i^{\ast}(A))$.

It is well known (see for instance  \cite[Section 1.2]{Facchini}) that if a ring $S$ is semilocal, so are all full matrix rings over $S$, all rings of the form  $eSe$ for an idempotent element $e\in S$, and all quotient rings of $S$. Further, direct products of finitely many semilocal rings are
semilocal, too.

Now, since $B$ is semilocal, we infer that
 $\End_B(X_j$) is semilocal for all $j$, and the direct product
\mbox{$R'=\End_B(X_1) \times \dots \times \End_B(X_l)$} is so, too. The inclusion
$R \subset R'$ is a local homomorphism in the sense that it carries non-units
to non-units (or equivalently, $R$ is rationally closed in $R'$),
because the inverse of an $l$-tuple of isomorphisms satisfying the commutativity
conditions automatically satisfies the commutativity conditions as well.
Therefore, a result by Camps and Dicks \cite[Corollary 2]{CampsDicks}
implies that $R$ is semilocal, too. Then so is its quotient $\bar{R}$, and the claim is proven.

Now, the statement follows from Theorem \ref{surjepi}.
\end{proof}

\begin{rem}
  In the proof of Corollary \ref{stratifyingrec} we need to change the left and
  (in general also) the right hand terms of the recollement to get a
  stratifying one, while leaving the middle term $\Dcal(A)$ unchanged. The
  original and the modified recollement are in the same equivalence class of
  recollements of $\Dcal(A)$, according to the definition of equivalence of
  recollements in \cite[1.7]{AKL1}.
\end{rem}

\section{Constructing homological epimorphisms with stratifying
kernel}

If $\lambda$ fails to be a surjective homological epimorphism with a
stratifying kernel, one may try to replace $\lambda$ by a new homological
epimorphism with better properties.

{\em Can one change a homological epimorphism into a stratifying one in a way
compatible with a given recollement?}

A more precise formulation of this question is as follows:
Suppose $\lambda:A\ra B$ is a homological epimorphism.  Are there  rings
$A'$ and
$B'$ which are derived equivalent to $A$ and $B$, respectively, and a
stratifying
homological epimorphism $\lambda':A'\ra B'$
such that the following diagram commutes?
\[
\xymatrix{
D(\Mod B) \ar[d]^{\simeq} \ar[r]^{\lambda_*} & D(\Mod A)\ar[d]^{\simeq} \\
D(\Mod B') \ar[r]_{{\lambda'}_*} & D(\Mod A') }
\]

The results in Sections 2 and 3 suggest that
some restrictions need to be imposed on the setup.

For instance, one can use the following
connection with tilting theory from \cite{GL,AHS}: if $\lambda:A \ra B$ is an injective ring epimorphism such that $\Tor^1_A(B,B)=0$ and the  right $A$-module $B_A$ has projective dimension at most
one (which implies in particular that $\lambda$ is  homological), then the $A$-module $T:=B\oplus B/A$ is tilting.  Tilting modules arising in this way  are characterised by the existence of a  $T$-coresolution of $A$ of the form $0 \rightarrow A \rightarrow T_0
\rightarrow T_1 \rightarrow 0$ where $T_0, T_1 \in \Add(T)$ satisfy
$\Hom_A(T_1,T_0)=0$, see \cite[Theorem 3.10]{AHS}.

Notice that such $T$ will not be finitely generated in general. Assuming $B_A$ to be finitely presented, however,
 gives a setup
of interest in our context, since $T$ is then a classical tilting module and $A$ can be replaced by a derived equivalent ring $A'$.

This will be our first construction. The second construction will produce from
$\lambda$ a new ring homomorphism $\mu: A \rightarrow C$, which will
be a homological epimorphism under suitable assumptions.

\medskip

{\em First construction.}

We present a case where the  question above has a positive answer. In fact, it will be sufficient to
change the ring $A$, while keeping $B$ unchanged.

\begin{prop}
  Suppose $\lambda:A\ra B$ is a homological epimorphism. If $\lambda$ is
  injective and $B_A$ is finitely presented of projective dimension at most
  one, then there are a ring $A'$ which is derived equivalent to $A$ and a
  surjective homological
  epimorphism $\lambda':A' \ra B$, such that the two
  epimorphisms induce  equivalent  recollements.
\end{prop}

\begin{proof}
  Under the assumptions made, $T:=B\oplus B/A$ is a tilting $A$-module and
  \mbox{$\Hom_A(B/A,B)=0$,} by \cite[Theorem 3.5]{AHS} and
  \cite[Proposition 4.12]{GL}. The homological epimorphism $\lambda$ induces a recollement
\[
\xy (-45,0)*{\mathcal \Dcal(B)}; {\ar (-25,2)*{}; (-35,2)*{}};
{\ar (-35,0)*{}; (-25,0)*{}^{}}; {\ar(-25,-2)*{}; (-35,-2)*{}};
(-15,0)*{{\mathcal D}(A)}; {\ar (5,2)*{}; (-5,2)*{}_{}};
{\ar (-5,0)*{}; (5,0)*{}}; {\ar (5,-2)*{}; (-5,-2)*{}};
(15,0)*{{\mathcal \Dcal(C)}};
\endxy
\]
where $C:=\End_A(B/A)$ (see \cite[Example 3.1]{AKL1} and  \cite[Theorem B]{MV}).
Moreover, $$A':=\End_A(T)=\begin{pmatrix} B=\End_A(B) & \Hom_A(B,B/A)\\
0 & C=\End_A(B/A) \end{pmatrix}$$ is derived equivalent to $A$.
This is a well studied situation,
see for instance \cite[Cor. 12 and 15]{Koenig91}.
The $T$-resolution of $A$
is $0\ra A\xrightarrow{\lambda} B\ra B/A\ra 0$.
Let $e\in A'$ be the idempotent corresponding to $B/A$.
Since $\Hom_A(B/A,B)=0$, $eA'(1-e) = 0$. Hence $eA'=eA'e=C$, $A'e=A'eA'$ and
$A'e\lten_{eA'e}eA'=A'eA'$, that is, the ideal $A'eA'$ generated by $e$ is stratifying. By construction, $A'/A'eA'
=\End_A(B)= B$. Hence the stratifying ideal $A'eA'$ induces a recollement
\[
\xy (-45,0)*{\mathcal \Dcal(B)}; {\ar (-25,2)*{}; (-35,2)*{}};
{\ar (-35,0)*{}; (-25,0)*{}^{}}; {\ar(-25,-2)*{}; (-35,-2)*{}};
(-15,0)*{{\mathcal D}(A')}; {\ar (5,2)*{}; (-5,2)*{}_{}};
{\ar (-5,0)*{}; (5,0)*{}}; {\ar (5,-2)*{}; (-5,-2)*{}};
(15,0)*{{\mathcal \Dcal(C)},};
\endxy
\]
which is equivalent to the original one. In particular, there is the desired commutative diagram of derived categories, involving the derived equivalence between $A$ and $A'$.
\end{proof}

The following example illustrates how the injective homological epimorphism $\lambda$ gets enlarged to
obtain a surjective homological epimorphism $\lambda'$ which is `derived equivalent' to $\lambda$ in the above sense.

\begin{ex}
Let $A$ be
the path algebra of the quiver $A_2$ over a field $k$; in other words, $A$ is the algebra of $2\times 2$ upper triangular matrices over $k$. Let
$B$ be the algebra of all $2\times 2$ matrices over $k$.
The inclusion $$\lambda:A=\begin{pmatrix} k & k \\ 0 & k \end{pmatrix} \hookrightarrow
\begin{pmatrix} k & k \\ k & k \end{pmatrix}=B$$ is a homological epimorphism such that the simple $B$-module
get identified with the projective-injective $A$-module $P$. So $B$ as a right $A$-module is isomorphic to $P\oplus P$.
 By \cite[Theorem 5.1]{AKL3}, $\lambda$ induces a recollement of $\Dcal(A)$ in terms of $\Dcal(B)$ and $\Dcal(C)$ for some  $k$-algebra $C$.
  Moreover, the image of $j_!$ (respectively, $j_*$) is generated by the
  simple injective (respectively, the simple projective) $A$-module, which is
  left (respectively, right) perpendicular to $P$.
This is an easy example of a recollement not of `stratifying type'. In fact, $A$ has two non-trivial stratifying ideals, generated by the two primitive idempotents, and the resulting recollements are different from the current one, as can be checked directly on objects.

The tilting module
$T:=B\oplus B/A$ is the direct sum $P\oplus P\oplus S$. Hence
$$A'=\End_A(T)=\begin{pmatrix} k& k & k\\ k & k & k \\ 0 & 0& k \end{pmatrix}$$
which is an enlarged version of $A$; $A$ and $A'$ are Morita equivalent, but $T$ is not a progenerator. The new homological epimorphism $\lambda':A'\ra B$
is surjective with a stratifying kernel.

This example also shows that modifying $B$ while keeping $A$ does in general not allow for a solution of the modification problem.
\end{ex}

{\em Second construction.}

The following general  construction of a ring homomorphism, whose kernel and cokernel can be controlled, will be  used  to produce homological epimorphisms with stratifying kernel.

We start with a ring homomorphism $f: A \rightarrow B$ with cone $K_f$  in $\Dcal(A)$, so that
there is a triangle in $\cd(A)$
$$(\dagger) \phantom{xx} A \stackrel{f}{\rightarrow} B \rightarrow K_f
\rightarrow A[1].$$

Denote by $C$ the endomorphism ring of $K_f$ in $\cd(A)$. Then a ring homomorphism
$\mu: A \rightarrow C$ can be defined as follows: any element $a \in A$ defines a module homomorphism $A \rightarrow A$ and its $f$-image $f(a)$ defines a module
homomorphism $B \rightarrow B$ according to
 the diagram
$$\begin{array}{ccc}
A & \stackrel{f}{\longrightarrow} & \hspace{-0.5cm} B \\
\phantom{x} \downarrow{a} & & \hspace{-0.5cm} \phantom{xxx} \downarrow{f(a)} \\
A & \stackrel{f}{\longrightarrow} & \hspace{-0.5cm} B
\end{array}$$
which is commutative, since $1_A$ gets sent to $f(a)$ in both ways. Therefore, the pair $(a,f(a))$ is an endomorphism of the complex $A \rightarrow B$ and induces an endomorphism of $K_f$ in $\cd(A)$. In this way, we obtain a ring homomorphism  $\mu:A\to C$.

\begin{thm} \label{secondconstruction}
Let $f:A\to B$ be a ring epimorphism whose cone $K_f$ satisfies $\Ext^{-1}_{\cd(A)}(K_f,K_f) := \Hom_{\cd(A)}(K_f,K_f[-1])= 0$. Assume that $\Tor^A_1(B,B)=0$. Then the ring homomorphism $\mu:A\to C$ defined above has kernel  $\Hom_A(B,A)$ and cokernel  $\Ext^1_A(B,A)$.
\end{thm}
\begin{proof}
Applying various $\Hom$-functors to the triangle $(\dagger)$ yields long exact
sequences, where we write $(X,Y)=\Hom_{\cd(A)}(X,Y)$ for short:
\begin{align*}
&(1) \hspace{0.5cm}
0=(A,A[-1]) \rightarrow (K_f,A) \rightarrow (B,A) \rightarrow (A,A) \xrightarrow{conn}
(K_f,A[1]) \rightarrow (B,A[1]) \rightarrow (A,A[1])=0,\\
&(2) \hspace{0.5cm}
0=(A,B[-1]) \rightarrow (K_f,B) \rightarrow (B,B) \stackrel{\alpha}{\rightarrow}
(A,B) \rightarrow (K_f,B[1]) \rightarrow (B,B[1]) = \Ext^1_A(B,B),\\
&(3) \hspace{0.5cm}
0=(K_f,K_f[-1]) \rightarrow (K_f,A) \rightarrow (K_f,B) \rightarrow (K_f,K_f)=C
\xrightarrow{\beta} (K_f,A[1]) \rightarrow (K_f,B[1]).\end{align*}

In $(1)$ and in $(2)$, the starting terms vanish, since modules don't have extensions in negative degrees. The starting term in $(3)$ vanishes by assumption on $K_f$.

In $(2)$,
the assumption $\Tor^A_1(B,B)=0$  implies $\Ext^1_A(B,B) \simeq \Ext^1_B(B,B)=0$ by \cite[Theorem 4.8]{Sch}. Since $f$ is a ring epimorphism, $\Hom_A(B,B) = \Hom_B(B,B) \simeq B \simeq \Hom_A(A,B)$, hence the map $\alpha$ is an isomorphism  and $(K_f,B) = 0 = (K_f,B[1])$.

Plugging this into $(3)$ gives $(K_f,A) = 0$ and $\beta: C \xrightarrow{\sim} (K_f,A[1])$.

Now the sequence $(1)$ reduces to
\[
0 \rightarrow (B,A) \rightarrow
(A,A) \stackrel{conn}{\rightarrow} (K_f,A[1]) \rightarrow (B,A[1])
\rightarrow 0.
\]
For any $a\in A$, the following commutative diagram
\[
\xymatrix{
A\ar[r]^{f} \ar[d]_{a} & B \ar[r] \ar[d]_{f(a)} & K_f \ar[r]^{\pi} \ar[d]_{\mu(a)} & A[1] \ar[d]_{a[1]} \\
A \ar[r]^{f} & B \ar[r] & K_f \ar[r]^{\pi} & A[1]
}
\]
shows that $conn(a) = a[1]\circ \pi=\pi\circ\mu(a)=(\beta\circ\mu)(a)$. Namely, under the identifications
$(A,A) \simeq A$ and $\beta: C\xrightarrow{\sim} (K_f,A[1])$, the connecting homomorphism $conn$
gets identified with $\mu$. This finishes the proof.
\end{proof}

In the special case of injective ring epimorphisms, $K_f = B/A$ is a module and thus it has no negative self-extensions. So Theorem \ref{secondconstruction} has the following consequence, generalising a construction from \cite[p.~295]{GL}.

\begin{cor} \label{injepimorphism}
 Let  $\lambda: A \ra B$ be an injective ring epimorphism such that
$\Tor_1^A(B,B)=0$. Let $C:=\End_A(B/A)$ be the endomorphism ring of
$B/A$ as right $A$-module. Then the  left $A$-module structure on $B/A$ induces a ring homomorphism $\mu:A\ra
C$ such that $\Ker(\mu)\simeq\Hom_A(B,A)$ and
$\Coker(\mu)\simeq\Ext^1_A(B,A)$.
\end{cor}

 Under additional assumptions, this leads to homological epimorphisms with stratifying kernel:

\begin{cor}\label{artinalg} Let $A,B$ be artin algebras,  and let  $\lambda: A \ra B$ be an injective ring epimorphism such that $B_A$ has projective dimension at most one and
$\Tor_1^A(B,B)=0$.
Then
$\mu:A\ra C$ is a homological epimorphism, and the projective dimension of $_AC$ as left $A$-module is at most one. Moreover, if also the projective dimension of $_A\Ext^1_A(B,A)$ as left $A$-module is at most one, then $\mu$ has a   stratifying kernel.
\end{cor}
\begin{proof}
The first statement follows from \cite[Proposition 4.13]{GL}. For the second statement, we set $I=\Ker(\mu)$ and consider the surjective ring epimorphism $\nu:A\to A/I$. The condition  $\Tor_i^A(A/I,A/I)=0$ for $i=1$ is verified  as in the first part of the proof of \cite[Lemma 4.5]{BS}, using that the left $A$-modules $C$ and $\Coker\mu$ have projective dimension at most one, and similarly one checks the cases $i\ge 2$.  So $\nu$ is a surjective homological epimorphism, and the claim follows from  Lemma~\ref{surjepitor} since  $A$ is perfect.
\end{proof}

\begin{ex}
  Let   $\lambda: A \ra B$   be an injective homological epimorphism of
  hereditary artin algebras. From the Corollaries above we deduce that the homological epimorphism $\mu:A\ra C$ is
\begin{enumerate}
\item[(i)] surjective with a stratifying kernel if and only if $B_A$ is projective,
\item[(ii)] injective  if and only if $B_A$ has no projective direct summand.
\end{enumerate}
For case (i) see also \cite[Theorem B]{MV}.

More concretely, let $A$ be the Kronecker algebra over a field $k$, and let $P_i$ be the indecomposable preprojective module of dimension vector $(i,i+1)$ for $i=1,2,3$. Consider the tilting module $T=P_1\oplus P_2$. The minimal $T$-coresolution of $A$ is given by $0\to A\to P_1\,^3\to P_2\to 0$, and $T$ arises from the
injective homological epimorphism $\lambda:A\to B=\End_A(P_1\,^3)$ as explained at the beginning of this Section. In this case $B_A$ is projective, $C=\End_A(P_2)\simeq k$, and $\mu:A\to C$ is the stratifying epimorphism induced by the idempotent element $e$ of $A$ corresponding to the projective module $P_1$.

Let us now consider  the tilting module $T'=P_2\oplus P_3$. The minimal
$T'$-coresolution of $A$ is given by $0\to A\to P_2\,^5\to P_3\,^3\to 0$,
and $T'$ arises
from the injective homological epimorphism $\lambda':A\to B'=\End_A(P_2\,^5)$.
Here $B'_A$ has no projective summand, $C=\End_A(P_3\,^3)\simeq M_3(k)$, and
$\mu:A\to C$ is injective.
\end{ex}


Finally we remark that in the situation of Corollary \ref{artinalg} there is a ladder of height two as follows
\[
\xy (-45,0)*{\mathcal \Dcal(B)}; {\ar (-25,2)*{}; (-35,2)*{}};
{\ar (-35,0)*{}; (-25,0)*{}^{}}; {\ar(-25,-2)*{}; (-35,-2)*{}}; {\ar(-35,-4)*{}; (-25,-4)*{}};
(-15,0)*{{\mathcal D}(A)}; {\ar (5,2)*{}; (-5,2)*{}_{}};
{\ar (-5,0)*{}; (5,0)*{}}; {\ar (5,-2)*{}; (-5,-2)*{}}; {\ar (-5,-4)*{}; (5,-4)*{}};
(15,0)*{{\mathcal \Dcal(C)}};
\endxy
\]
where the `upper' and `lower' recollements are induced by the homological
epimorphisms \mbox{$\lambda: A \ra B$} and
$\mu: A \ra C$ respectively (for the terminology of ladder see \cite{AKLY2}). This shows that in some cases our second construction in Theorem C changes a homological epimorphism into a stratifying one, such that the induced recollements are not equivalent but lie on the same ladder. More precisely, the original recollement can be reflected one step downward to get the new one.


\begin{thebibliography}{99}
\bibitem{AKL1} {\sc L. Angeleri H\"ugel, S. Koenig and Q. Liu,}
  \emph{Recollements
and tilting objects}, J. Pure. Appl. Algebra {\bf 215} (2011), 420--438.
Also arXiv:0908.1988.

\bibitem{AKL3} {\sc L. Angeleri H\"ugel, S. Koenig and Q. Liu,}
\emph{Jordan-H\"older theorems for derived module categories of piecewise
hereditary algebras}, J. Algebra {\bf 352} (2012), 361--381. Also
arXiv:1104.3418.

\bibitem{AKLY2} {\sc L. Angeleri H\"ugel, S. Koenig, Q. Liu and D. Yang,}
\emph{Ladders and simplicity of derived module categories},  arXiv:1310.3479.

\bibitem{AMV}{\sc L. Angeleri H\"ugel, F. Marks and J. Vit\'oria,}
  \emph{Silting modules}, Int. Math. Res. Not. IMRN 2016, no. 4, 1251--1284,
  Also arXiv:1405.2531.

\bibitem{AHS}{\sc L. Angeleri H\" ugel and J. S\'anchez},
\emph{Tilting modules arising from ring epimorphisms},
Algebr. Represent. Theory {\bf 14} (2011), 217--246.

\bibitem{BS} {\sc S.~Bazzoni, J.~\v{S}\v{t}ov\'\i\v{c}ek},\emph{Smashing localisations of rings of weak global dimension at most one},preprint, arXiv:1402.7294.

\bibitem{BBD}
{\sc A. A. Beilinson, J. Bernstein and P. Deligne,}
\emph{Analyse et topologie sur les espaces singuliers},
{Ast{\'e}risque}, vol. {\bf 100}, {Soc. Math. France}, 1982.

\bibitem{BD}{\sc G. Bergman, W. Dicks}, \emph{Universal derivations and universal ring constructions}. Pacific J. Math.{\bf 79} (1978), 293-337.

\bibitem{CampsDicks}{\sc R. Camps and W. Dicks,} {\it On semilocal rings},
Israel J. Math. {\bf 81} (1993), 203--211.

\bibitem{CX}{\sc H. X. Chen and C. C. Xi}, {\it Homological ring epimorphisms
  and recollements II: Algebraic K-theory}, arXiv:1212.1879.

\bibitem{CPS88}{\sc E. Cline, B. Parshall and L. Scott},
  {\sl Finite-dimensional algebras and highest weight categories}, J. reine ang.
  Math. {\bf 391} (1988), 85--99.

\bibitem{CPS88b}\bysame,
  {\sl Algebraic stratification in representation categories}, J. Algebra
  {\bf 117} (1988), 504--521.

\bibitem{CPS}\bysame, {\sl
  Stratifying endomorphism algebras}. {Mem. Amer. Math. Soc. }{\bf 124},
  (1996), no. 591.


\bibitem{Facchini} {\sc A. Facchini,}
  \emph{Module theory. Endomorphism rings and direct sum decompositions in some
    classes of modules.} Progress in Mathematics, 167. Birkh\"auser Verlag,
  Basel, 1998. xiv+285 pp.

\bibitem{GL}{\sc W. Geigle and H. Lenzing}, \emph{Perpendicular categories with
applications to representations and sheaves}, J. Algebra  {\bf 144}
(1991),  no. 2, 273--343.

\bibitem{Han11} {\sc Y. Han}, \emph{Recollements and Hochschild theory}, J. Algebra {\bf 397} (2014), 535--547.


\bibitem{Happel93} {\sc D. Happel}, \emph{Reduction techniques for homological
  conjectures}, Tsukuba J. Math. {\bf 17} (1993), no. 1, 115--130.



\bibitem{IyamaYang14} {\sc O. Iyama and D. Yang}, \emph{Silting reduction and Calabi--Yau reduction of triangulated categories}, arXiv:1408.2678.

\bibitem{KalckYang}{\sc M. Kalck and D. Yang}, \emph{Relative singularity categories I: Auslander resolutions}, arxiv:1205.1008.

\bibitem{Kasch}{\sc F. Kasch}, {\sl Modules and rings}, Academic Press 1982.

\bibitem{Keller} {\sc B. Keller,} \emph{Deriving DG categories},
Ann. Sci. \'Ec. Norm. Sup\'er. (4) {\bf 27}, 63--102 (1994).


\bibitem{Keller98}
\bysame, \emph{Invariance and localization for cyclic homology of
  {D}{G} algebras}, J. Pure Appl. Algebra \textbf{123} (1998), no.~1-3,
  223--273.

\bibitem{Koenig91}
{\sc S. Koenig,} \emph{Tilting complexes, perpendicular categories and
recollements of derived categories of rings}, J. Pure Appl. Algebra
{\bf 73} (1991), 211--232.

\bibitem{KoenigNagase09}
{\sc S. Koenig and H. Nagase,} \emph{Hochschild cohomologies and
  stratifying ideals}, J. Pure Appl. Algebra \textbf{213} (2009), no.~4,
  886--891.

\bibitem{LVY} {\sc Q. Liu, J. Vit\'oria and D. Yang,}
  \emph{Gluing silting objects}, Nagoya Math. J. {\bf 216} (2014), 117--151.
  Also arXiv:1206.4882.

\bibitem{MV} {\sc F. Marks and J. Vit\'oria}, \emph{From ring epimorphisms to
  universal localisations}, Forum Math. {\bf 27} (2015), 1139--1161.
  Also arXiv:1207.4669.

\bibitem{Michler}{\sc G. Michler}, \emph{Idempotent ideals in perfect rings}, Canad. J. Math. {\bf 21} (1969), 301--309.

\bibitem{NZ} {\sc P. Nicol\'as,} \emph{On torsion torsionfree triples},
PhD Thesis, Murcia 2007.

\bibitem{NS1} {\sc P. Nicol\'as and M. Saor\'in,}  \emph{Parametrizing
  recollement data for triangulated categories}, J. Algebra. {\bf 322} (2009),
  1220--1250.

\bibitem{PV} {\sc C. Psaroudakis and J. Vit\'oria,} \emph{Recollements of module
categories}, Appl. Categ. Structures {\bf 22} (2014), 579--593.

\bibitem{PVpositive} {\sc C. Psaroudakis and J. Vit\'oria,}
  \emph{Realisation functors in tilting theory}, arXiv:1511.02677.

\bibitem{Reid}{\sc G. A. Reid}, \emph{Epimorphisms and surjectivity},
  Invent. Math. {\bf 9} (1970), 295-307.

\bibitem{RZ}{\sc R. Rouquier and A. Zimmermann,} \emph{Picard groups for
  derived module categories}, Proc. London Math. Soc. (3)
  {\bf 87} (2003), 197--225.

\bibitem{Schlichting06}
{\sc M. Schlichting}, \emph{Negative {$K$}-theory of derived categories}, Math. Z. \textbf{253} (2006), 97--134.

\bibitem{Sch} {\sc A.~H.~Schofield}, {\em Representations of rings over skew fields}, Cambridge University Press (1985).

\bibitem{Ste} {\sc B.~Stenstr\"om}, {\em Rings of Quotients}, Springer-Verlag (1975), viii+309pp.

\bibitem{Storrer}  {\sc H. Storrer,} \emph{Epimorphismen von kommutativen
  Ringen},
Comment. Math. Helv. {\bf 43} (1968) 378--401.

\bibitem{Tabuada05}
{\sc G. Tabuada}, \emph{Une structure de cat\'egorie de mod\`eles de Quillen sur la cat\'egorie des dg-cat\'egories}, C. R. Math. Acad.
Sci. Paris {\bf 340} (1) (2005) 15--19.

\bibitem{Weibel} {\sc C. Weibel,} {\sl An introduction to homological algebra.}
Cambridge Studies in Adv. Mathematics Vol. 38. Cambridge Univ. Press 2008.

\bibitem{Zimmermann} {\sc A. Zimmermann,} {\sl Representation theory. A
homological algebra point of view.} Algebra and Applications Vol. 19. Springer
2014.
\end{thebibliography}
\end{document}